\documentclass[12pt,reqno]{amsart}

\usepackage{amsthm,amsfonts,amsmath,amssymb}
\usepackage[pagebackref=true,backref=true,bookmarksopen=true]{hyperref}

\newtheorem{thm}{Theorem}
\DeclareMathOperator{\Div}{div}

\begin{document}
\title[Second order regularity]{Second order regularity of solutions of elliptic equations in divergence form with Sobolev coefficients}
\author{M.A. Perelmuter}
\address{SCAD Soft Ltd., 3a Osvity street, Kyiv, 03037, Ukraine}
\email{mikeperelmuter@gmail.com}

\maketitle

\begin{abstract}
We give $L^p$ estimates for the second derivatives of weak solutions to the Dirichlet problem for equation $\Div(\mathbf{A}\nabla u) = f$ in $\Omega\subset \mathbb{R}^d$ with Sobolev coefficients.
In particular, for $f\in L^2(\Omega) \bigcap L^s(\Omega)$
$$\|\Delta u\|_{2} \leq \begin{cases}
c_1\|f\|_2 + c_2 \|\nabla \mathbf{A}\|_q^2\|f\|_s,  & \text{if } 1 < s < d/2, \frac{1}{2}=\frac{2}{q}+ \frac{1}{s} - \frac{2}{d}\\
c_1\|f\|_2 + c_2 \|\nabla \mathbf{A}\|_4^2\|f\|_s,  & \text{if } s > d/2
\end{cases}.$$
\end{abstract}

\section{Introduction}
Let $\Omega \subset {\mathbb R}^d, d\geq 2$ be an open bounded domain with $C^2$ boundary. Let $\mathbf{A}:\Omega \rightarrow \mathbb{R}^{d\times d}$ be a measurable real-valued symmetric matrix
that satisfies the ellipticity conditions
\begin{equation} \label{ELower}
\lambda |\xi|^2 \leq\langle \mathbf{A}\xi,\xi \rangle , \qquad  \forall \xi \in \mathbb{R}^d.
\end{equation}
\begin{equation} \label{EUpper}
\langle \mathbf{A}\xi,\xi \rangle \leq \Lambda |\xi|^2, \qquad  \forall \xi \in \mathbb{R}^d.
\end{equation}
In this paper, we study the $L^p$   regularity of solutions to the following elliptic equation in divergence form with discontinuous coefficients
\begin{equation} \label{main1}
\begin{cases}
\sum\limits_{i,j=1}^d \frac{\partial}{\partial x_i} \left(a_{ij}(x) \frac{\partial u}{\partial x_j}\right) \equiv \Div(\mathbf{A}\nabla u) = f & \text{in } \Omega\\
\;\;u = 0  &  \text{on } \partial \Omega
\end{cases}.
\end{equation}
We denote the norm in $L^p(\Omega)$ by $\|\cdot\|_p$ and define the Sobolev space $W^{1,p}_0(\Omega)$ as the closure of $C_0^\infty(\Omega)$ with respect to the norm
$$\|v\|_{1,p} = \left(\int\limits_\Omega |\nabla v|^p \mathrm{d} x\right)^{1/p}. $$
Consider the sesquilinear form
$$\mathfrak{a}: W^{1,2}_0(\Omega) \times W^{1,2}_0(\Omega)\to {\mathbb C},  \mathfrak{a}(u,v) = \int\limits_\Omega \mathbf{A}\nabla u \cdot \nabla v \mathrm{d} x. $$
The ellipticity conditions \eqref{ELower}-\eqref{EUpper} imply that $\mathfrak{a}$ is coercive and bounded.
Using the Lax-Milgram theorem \cite[Chapter III.7]{Yosida} we can associate with the problem \eqref{main1} operator $\mathcal{L}$ in $L^2(\Omega)$ through this sesquilinear form. $\mathcal{L}$ is self-adjoint operator, $\mathcal{L}^{-1}$ is a bounded operator in $L^2(\Omega)$, $\mathcal{L}^{-1}$ maps $W^{-1,2}$ continuously into $W^{1,2}_0$, and
\begin{equation}\label{Meyers0}
\|\nabla\mathcal{L}^{-1}\nabla u \|_2 \leq \lambda^{-1} \|u\|_2.
\end{equation}
The problem of first-order regularity ($W^{1,p}$ estimate for solutions) was solved by N.G.Meyers~\cite{M}. Meyers considered equation $\Div\mathbf{A}\nabla u = \Div \mathbf{F} +f$ on a bounded domain with $C^1$ boundary. He showed that if matrix $\mathbf{A}$ is uniformly elliptic, then there exists $p_0 = p_0(d,\lambda, \Lambda) > 2$  such that for all $p_0' < p < p_0$
(here $p_0'=\frac{p_0}{p_0-1}$ is the H\"older conjugate exponent) there exists a weak solution $u$ that satisfies
\begin{equation}\label{Meyers1}
  	\| \nabla u\|_p \leq C(\|\mathbf{F}\|_p + \|f\|_p).
\end{equation}
The largest possible number $p$ in \eqref{Meyers1} is called Meyers exponent and is denoted by $p_0$.
Finding the Meyers exponent $p_0$ for $d > 2$ is an open problem. For $d=2$  it is known that $p_0 = \frac{2K}{K-1}, K=\sqrt{\frac{\Lambda}{\lambda}}$  \cite{LN} and this result is optimal (see example in \cite{M}). The $d$-independent bound $p_0 \leq \frac{14K-12}{7K-7}$ was proved by T.Iwaniec and C.Sbordone~\cite[Theorem 2]{IS2001}. Therefore, we can assume that $p_0=p_0(\lambda,\Lambda)$ depends on the ellipticity constants only. In what follows we will denote by $I_M$ the interval $(p_0',p_0)$.
\par
Study of the second order regularity of solutions to linear divergent equations ($W^{2,p}$ estimate for solutions) with
discontinuous coefficients goes back to C.Miranda~\cite[Theorem 1.VIII]{Miranda}, who  showed
that if $d\geq 3$ and $\mathbf{A}$ belongs to the Sobolev class $W^{1,d}(\Omega)$, then any weak solution of \eqref{main1} with
$f\in L^2(\Omega)$ is a strong solution and
$\|D^2u \|_2 \leq C\|f\|_2,$
where $D^2=\sum\limits_{i,j=1}^d D_i D_j$.
Here and later, we denote the partial derivatives as $D_j u(x) = \frac{\partial u}{\partial x_j} (x), j=1,\ldots, d$.
\par
$W^{2,p}$-version of the Miranda's theorem is proved in~\cite{CMR2016}. Namely, the authors of~\cite{CMR2016} proved that for all
  $p\in I_M \bigcap (1,2)$ and $f\in L^p(\Omega)$ there exists a unique solution $u$ of \eqref{main1}
  that satisfies the following estimate $ \|D^2 u \|_p \leq C\|f\|_p$. If $d \geq 3$, it is possible to take $p\in I_M \bigcap (1,2]$.
\par
The aim of this paper is to prove that it is possible to obtain a bound for the solution
in $W^{2,p}$ of the problem \eqref{main1} assuming higher integrability of the right hand side $f$ and lower integrability of $\nabla \mathbf{A}$.
\par
Our proof uses commutator trick from \cite{PerSem1985}, Gagliardo-Nirenberg inequality, and Meyers regularity result. Note that the methods used in the proofs allow us to obtain explicit values of constants in the estimates.

\section{Main results}
In part \rm{\bf{II}} of Theorem~\ref{Theorem}, the entries $a_{ij}$ of the matrix $\mathbf{A}$ may be unbounded.
In this situation, the construction of the operator  $\mathcal{L}$ described in the introduction cannot be used, since it is based on the Lax-Milgram theorem. However, we can consider a sequence of matrices $\mathbf{A}_n, n=1,\ldots$ with coefficients
$$a_{ij}^{(n)}= \lambda \delta_{ij} + \left(a_{ij}-\lambda\delta_{ij}\right)\left(1+\frac{1}{n}\sum\limits_{j=1}^d a_{jj}\right)^{-1}.$$
It is straightforward to verify that $a_{ij}^{(n)}\in L^\infty(\Omega)$
\footnote{The boundedness of the diagonal elements follows from the definition. For off-diagonal elements, due to the symmetry and positive definiteness of the matrix $|a_{ij}| \leq (a_{ii} + a_{jj})/2.$}, that $\mathbf{A}_n$ satisfies condition \eqref{ELower} with the same constant $\lambda$, and that $\langle \mathbf{A}_n\xi,\xi \rangle \leq \langle \mathbf{A}_{n+1}\xi,\xi \rangle$. Let $\mathcal{L}_n$ denote the operator corresponding to the matrix $\mathbf{A}_n$.
Since the corresponding sesquilinear forms $\mathfrak{a}_n$ are monotonically increasing, then
the limit form $\mathfrak{a}=\lim\limits_{n\to\infty} \mathfrak{a}_n$ is densely defined,
closed, and for the limit operator  $\mathcal{L}$ associated with $\mathfrak{a}$  there is strong resolvent convergence  $\mathcal{L} = \lim\limits_{n\to\infty}\mathcal{L}_n$ (see \cite{Simon1}, \cite{Simon2}).
We use this operator in the case of unbounded coefficients.

\begin{thm}\label{Theorem}
Let $\Omega \subset {\mathbb R}^d, d\geq 2$ be an open bounded domain with $C^2$ boundary.
Let $\mathbf{A}:\Omega \rightarrow \mathbb{R}^{d\times d}$ be a measurable real-valued symmetric matrix  and $\mathbf{A} \in W^{1,q}(\Omega)$.
Let $u$ be the solution to equation \eqref{main1}.
\begin{itemize}
  \item[\rm{\bf{I.}}] Assume that the matrix $\mathbf{A}$ satisfies conditions \eqref{ELower}-\eqref{EUpper} and that $f\in L^p(\Omega) \bigcap L^s(\Omega)$, $p\in I_M$. Then $u\in W^{2,p}(\Omega)$ and
        \begin{equation}\label{Pp}
        \sum\limits_{i,j=1}^d \left\|D_i D_j u\right\|_{p} \leq C_1 \|f\|_p + C_2 \left(\sum\limits_{k,i,j=1}^d \|D_k a_{ij}\|_q\right)^2\|f\|_s,
        \end{equation}
    where $\frac{1}{p}=\frac{2}{q}+ \frac{1}{s} - \frac{2}{d}$ if $1 < s < d/2$ and $q=2p$ if $s>d/2$.
  \item[\rm{\bf{II.}}] Assume that matrix $\mathbf{A}$ satisfies condition \eqref{ELower} and that $f\in L^2(\Omega) \bigcap L^s(\Omega)$. Then $u\in W^{2,2}$ and
        \begin{equation}\label{P2}
        \sum\limits_{i,j=1}^d \left\| D_i D_j u\right\|_{2} \leq \frac{C_1}{\lambda}\|f\|_2 + \frac{C_2}{\lambda^2}   \left(\sum\limits_{k,i,j=1}^d \|D_k a_{ij}\|_q\right)^2\|f\|_s,
        \end{equation}
    where $\frac{1}{2}=\frac{2}{q}+ \frac{1}{s} - \frac{2}{d}$ if $1 < s < d/2$ and $q=4$ if $s>d/2$.
\end{itemize}
The constants $C_1$ and $C_2$ are independent of $f$ and $u$.
\end{thm}
\begin{proof}
Approximation arguments allow to reduce the proof to the case of smooth functions and smooth boundary (see below).
Thus, let assume that $a_{ij}\in C^\infty(\Omega)$ and $f\in C^\infty(\Omega)$.
By elliptic regularity theory, the solution $u$ in this case belongs to $C^\infty(\Omega)$ (see \cite[Theorem 3 in \S 6.3.1]{Evans}).
\par
A standard theorem in harmonic analysis states that $\left\|D_i D_j \varphi\right\|_{L^p(\mathbb{R}^d)} \leq A_p \|\Delta \varphi\|_{L^p(\mathbb{R}^d)}, 1 < p < \infty$, $i,j = 1,\ldots, d, \varphi\in C^2_0(\mathbb{R}^d)$ (see, e.g.,~\cite[\S 1.3 in Chapter III]{Stein}).
Therefore, we can estimate $\| \Delta u\|_p$ only.
\par
For smooth coefficients, we have the equality $[D_k, \mathcal{L}]v=\sum\limits_{i,j=1}^d D_i (D_k a_{ij})D_j v$, where $[\cdot,\cdot]$ denotes the
usual commutator bracket. To proceed with the proof, we also need the following commutator equality for linear operators: $ [A,B^{-1}] = B^{-1} [B,A] B^{-1}$.
\par
Now we have
\begin{equation}
\begin{aligned}
&\left\|\Delta\mathcal{L}^{-1}f\right\|_p \leq \sum\limits_{k=1}^d \left\|D_k D_k \mathcal{L}^{-1} f \right\|_p \leq \\
&\sum\limits_{k=1}^d \left(\left\|D_k \mathcal{L}^{-1} D_k f \right\|_p + \left\|D_k \mathcal{L}^{-1} [D_k, \mathcal{L}] \mathcal{L}^{-1} f\right\|_p\right) \leq \\
&\sum\limits_{k=1}^d \left\|D_k \mathcal{L}^{-1} D_k f \right\|_p  + \sum\limits_{i,j,k=1}^d\left\|D_k \mathcal{L}^{-1} D_i (D_k a_{ij}) D_j \mathcal{L}^{-1} u\right\|_p.
\end{aligned}
\end{equation}
Meyers' gradient estimate \eqref{Meyers1} implies that
\begin{equation}\label{Meyers2}
\left\|D_k \mathcal{L}^{-1} D_i  f\right\|_p \leq C_M\|f\|_p, \quad i,k=1,\ldots,d; \quad C_M=C_M(p,\lambda,\Lambda) < \infty.
\end{equation}
By \eqref{Meyers2} and H\"older's inequality we can estimate as follows:
$$\left\|\Delta\mathcal{L}^{-1}f\right\|_p \leq d C_M \|f\|_p + C_M \sum\limits_{i,j,k=1}^d \|D_k a_{ij}\|_q \cdot \|D_j \mathcal{L}^{-1} f  \|_r, \frac{1}{p}=\frac{1}{q}+\frac{1}{r}.$$
We now use the Gagliardo-Nirenberg inequality (see, e.g.,~\cite[Theorem 12.87]{Leoni})
\footnote{For bounded Lipschitz domains Gagliardo-Nirenberg inequality has the form $\|D_j \varphi\|_r \leq C_{GN} \|\Delta \varphi\|_p^{1/2} \cdot \|\varphi\|_t^{1/2} + c \|\varphi\|_t$ (see, e.g.,~\cite{LZ2022}), but for $\varphi\in W^{2,p}_0$ the constant $c$ can be equal to zero.
Although the presence of this term is not an obstacle to proof, it may be essential when considering other boundary conditions.}
\begin{equation*}\label{GN1}
\|D_j \varphi\|_r \leq C_{GN} \|\Delta \varphi\|_p^{1/2} \cdot \|\varphi\|_t^{1/2},
\end{equation*}
where $\frac{2}{r}=\frac{1}{p}+\frac{1}{t}$,  $C_{GN} = C_{GN}(d,p,t) < \infty$.
\par
The Cauchy-Bunyakovsky-Schwarz inequality implies that
\begin{equation}\label{GN2}
\|D_j \varphi\|_r \leq C_{GN}\left(\varepsilon\|\Delta \varphi\|_p + \frac{1}{4\varepsilon}\|\varphi\|_t\right), \quad  \varepsilon>0.
\end{equation}
From \eqref{GN2}, we now obtain
$$\left\|\Delta\mathcal{L}^{-1}f\right\|_p \leq d C_M \|f\|_p + C_M C_{GN} \sum\limits_{k,i,j=1}^d \|D_k a_{ij}\|_q \cdot \left(\varepsilon\|\Delta \mathcal{L}^{-1}f\|_p + \frac{1}{4\varepsilon}\|\mathcal{L}^{-1}f\|_t \right).$$
For sufficiently small $\varepsilon>0$

\begin{equation*}
\left\|\Delta\mathcal{L}^{-1}f\right\|_p \leq
\frac{d C_M \|f\|_p + \frac{C_M C_{GN}}{4\varepsilon} \|\mathcal{L}^{-1}f\|_t \sum\limits_{k,i,j=1}^d \|D_k a_{ij}\|_q}{1-\varepsilon C_M C_{GN}\sum\limits_{k,i,j=1}^d \|D_k a_{ij}\|_q}.
\end{equation*}

We choose $\varepsilon$ such that the denominator in the previous inequality equals $1/2$, i.e.
$$\varepsilon = \left(2C_M C_{GN}\sum\limits_{k,i,j=1}^d \|D_k a_{ij}\|_q\right)^{-1}. $$
Then
\begin{equation}\label{Estimate0}
\left\|\Delta\mathcal{L}^{-1}f\right\|_p \leq 2dC_M \|f\|_p + \frac{1}{2}C_M^2 C_{GN}^2 \left(\sum\limits_{k,i,j=1}^d \|D_k a_{ij}\|_q\right)^2\|\mathcal{L}^{-1}f\|_t.
\end{equation}
Let us now use the following result (see~\cite{Stamp},~\cite{Maz1961},~\cite{Trudin})
\begin{align*}
\left\|\mathcal{L}^{-1}f \right\|_t &\leq C_S\|f \|_{s}, &&\text{if } s < \frac{d}{2} \text{ and } \frac{1}{t}=\frac{1}{s}-\frac{2}{d},\\
\left\|\mathcal{L}^{-1}f \right\|_\infty &\leq C_S\|f\|_s, &&\text{if } s > \frac{d}{2}.
\end{align*}
If we combine this with \eqref{Estimate0}, we complete the proof of part \rm{\bf{I}} in the ''smooth'' case.

\medskip
\par
The reduction of the problem to the case of bounded coefficients was briefly described above.
The transition from bounded functions to smooth ones is well-known. We will not describe this process here and refer the reader to~\cite[Proof of Theorem 8.12]{GT}, \cite{AQ2002}, \cite{BE2024} for details. We only make just a few comments.
The possibility of using domains with infinitely smooth boundaries is based on the choice of an increasing sequence of $C^\infty$ subdomains $\Omega_n$ converging to $\Omega$.
A standard mollification may be used to smooth the coefficients.
Choose functions $\mu_n(x)=\mu(nx)\left(\int \mu(nx) {\rm d}x \right)^{-1}$, where $\mu\in C_0^\infty(\mathbb{R}^d)$, $0 \leq \mu \leq 1$, $\mu(x)=0$ for $|x|\geq 2$, and $\mu(x)=1$ for $|x|\leq 1$.
Define the smoothed coefficients by $a_{ij}^{(m)} = \mu_m * a_{ij}$. Matrices $\mathbf{A}^{(m)}$ have the same ellipticity constants $\lambda, \Lambda$ as the original matrix $\mathbf{A}$; operator $\mathcal{L}$ is limit in the sense of strong resolvent convergence of operators $\mathcal{L}^{(m)}$ with smooth coefficients.
\medskip
\par
To prove part \rm{\bf{II}}, it remains to note that when $p = 2$, we do not have to use \eqref{Meyers2}; we could have applied \eqref{Meyers0}.
Therefore, the results are correct even without the assumption of the boundedness of the coefficients $a_{ij}$  and constant $C_M$ equals $\lambda^{-1}$.
\end{proof}

{\bf Remarks.}
\begin{itemize}
  \item[(1)] {For } $q=d, s=2$ part \rm{\bf{II}} of Theorem \ref{Theorem} corresponds to Miranda's result, but without the assumption of boundedness of the coefficients (from $\mathbf{A}\in W^{1,d}(\Omega)$ boundedness does not follow).
  \item[(2)] For $p\in I_M$ and $s=p < 2$, part \rm{\bf{{I}}} of Theorem \ref{Theorem}  agrees with the results in \cite{CMR2016}.
  \item[(3)] For $p=2, s = \infty$, and $q=4$ we obtained result proved in \cite{PerSem1985} for the operator $\mathcal{L}+1$ in $L^2(\mathbb{R}^d)$ (see also \cite{GPN2019}).
\end{itemize}

\section{Final remarks}
Whether Theorem \ref{Theorem} is sharp for all values of the parameters $p,q,s$ remains an open question (sharpness means that weakening the assumption $\mathbf{A}\in W^{1,q}$ implies a loss of regularity). In the classical case of Miranda's theorem, where $f\in L^2(\Omega), \mathbf{A}\in W^{1,d}(\Omega)$, the result is sharp.
If $d=2$ and $\mathbf{A}\in W^{1,2}(\Omega)$, in papers \cite{Clop2009}, \cite{CMR2016}, using the theory of quasiconformal mappings, was constructed an example that demonstrates the failure of $W^{2,2}$-regularity in general.
\par
Let us show that in the multidimensional case, a similar negative result occurs as well.
Let $d > 2$ and assume that $u\in W^{2,2}_{loc}(\Omega)$. Then, by the Sobolev embedding theorem, $\nabla u\in L^{\frac{2d}{d-1}}_{loc}$.
However, N.G.Meyers~\cite{M} constructed an example of the equation $L_M u = \Div\mathbf{A}\nabla u=0$ with a solution such that $\nabla u \notin L^{p_0}_{loc}$, $p_0 = \frac{2K}{K-1}$, $K=\sqrt{\frac{\Lambda}{\lambda}}$.
In his example $\nabla\mathbf{A}\in L^{d,\infty}$ (Marcinkiewicz space) and the inclusion $u\in W^{2,2}_{loc}(\Omega)$  may fail if $\Lambda/\lambda$ is sufficiently large.
It follows that assumption $\mathbf{A}\in W^{1,d}$ is sharp.
\par
We can use this example and method of descent to show the sharpness of our estimates in some other cases.
Let $d=4, s=\infty, p=2$. Estimate \eqref{P2} reads as follows
$$\|u\|_{2,2} \leq c_1\|f\|_2 + c_2 \|\nabla \mathbf{A}\|_4^2\|f\|_\infty$$
and the above arguments show that this estimate may fail if $\nabla\mathbf{A}\in L^{4,\infty}(\mathbb{R}^4)$.
Now let $d > 4$, $\mathbf{x}\in \mathbb{R}^4$, $(\mathbf{x},\mathbf{y})$ be a point in $\mathbb{R}^d$ where
$\mathbf{y}$ stands for the remaining $(d - 4)$ coordinates.
We extend the given solution $u$ and coefficients by defining $u(\mathbf{x},\mathbf{y})=u(\mathbf{x})$, etc. Then we have $L_M u+\Delta_{\mathbf{z}} u=0$.\footnote{Here, we have repeated verbatim Meyers' arguments, who constructed an example in $\mathbb{R}^2$ and used the descent method to construct an example in $\mathbb{R}^d, d > 2$.} Obviously, this example demonstrates the sharpness of the condition $\mathbf{A}\in W^{1,4}(\mathbb{R}^d)$.

\par
The limiting case $s = d/2$ is not considered in Theorem \ref{Theorem}, but our method of proof allows one to do this and estimate $\|\Delta u\|_{p}$ via the Orlicz-space norm of $f$.
\par
The assumption of $C^2$-smoothness of the boundary cannot be significantly weakened. Already in the case of $L = \Delta$, $f\in C_0^\infty(\Omega)$ there exists $C^1$ domains such that $\Delta u \notin L^p(\Omega)$ for every $p$, $1 \leq p < \infty$~\cite[Theorem 1.2b]{JK1995}.
\par
For any open subset  $\Omega' \Subset \Omega$ with ${\rm dist}(\partial\Omega',\partial\Omega)>0$, \emph{local} estimates such as  \eqref{Pp} and  \eqref{P2} can be proved for $C^1$-domains.
\par
Having estimates of $\|\Delta u\|_p$, using the Gagliardo-Nirenberg inequality, we can obtain a higher integrability of $\nabla u$, for example, $\mathbf{A}\in W^{1,4}, f\in L^\infty \Rightarrow u\in W^{2,2}\bigcap L^\infty \subset W^{1,4}$.
\par
For simplicity, we present the results here for equations without lower-order terms, though the methods of this paper can be adapted to more general cases.
\par
Equations of the form $\Div(\mathbf{A}\nabla u) = \Div \mathbf{F}$ with $\mathbf{A}\in VMO$ (the space of functions of vanishing mean oscillation) were considered by
G.Di~Fazio~\cite{DiFazio} (see also \cite{AQ2002}, \cite{IS1998}).  It is proved the existence of the unique weak solution that satisfies $\|\nabla u \|_p \leq C\|\mathbf{F}\|_p$,  $1<p<\infty$.
Therefore, if  $\mathbf{A}\in VMO\bigcap W^{1,q}$, it is possible to assume $p\in (1,\infty)$ in part \rm{\bf{I}} of Theorem \ref{Theorem}.



\begin{thebibliography}{99}

\bibitem{AQ2002} Auscher,~P., Qafsaoui,~M.:
\newblock Observations on {$W^{1,p}$} estimates for divergence elliptic equations with VMO coefficients.
\newblock Boll. Unione Mat. Ital., Sez. B, Artic. Ric. Mat., \textbf{(8) 5}, No. 2, 487-509 (2002)

\bibitem{BE2024} B\"ohnlein,~T., Egert,~M.:
\newblock Explicit improvements for $L^p$-estimates related to elliptic systems.
\newblock Bull. Lond. Math. Soc., \textbf{56}, No. 3, 914-930 (2024)

\bibitem{Clop2009} Clop,~A., Faraco,~D., Mateu,~J., Orobitg,~J.,~Zhong,~X.:
\newblock Beltrami equations with coefficient in the Sobolev space $W^{1,p}$.
\newblock Publ. Mat. \textbf{53}, No. 1, 197-230 (2009)

\bibitem{CMR2016} Cruz-Uribe,~D., Moen,~K., Rodney,~S.:
\newblock Regularity results for weak solutions of elliptic PDEs below the natural exponent.
\newblock Annali di Matematica, \textbf{195}, No. 3, 725-740 (2016)

\bibitem{DiFazio} Di~Fazio,~G.:
\newblock {$L^p$} estimates for divergence form elliptic equations with discontinuous coefficients.
\newblock Boll. Un. Mat. Ital. Serie VII. A, \textbf{10}, No. 2, 409-420 (1996)

\bibitem{Evans} Evans,~L.~C.:
\newblock Partial differential equations, 2nd ed..
\newblock Graduate Studies in Mathematics, \textbf{19},  Providence, RI, American Mathematical Society, (2010)

\bibitem{GT} Gilbarg,~D., Trudinger,~N.~S.:
\newblock Elliptic Partial Differential Equations of Second Order.
\newblock Classics in Mathematics. Springer, Berlin (2001)

\bibitem{GPN2019} Giova,~R., Passarelli di Napoli,~A.:
\newblock Regularity results for a priori bounded minimizers of non-autonomous functionals with discontinuous coefficients.
\newblock Adv. Calc. Var. \textbf{12}, No. 1, 85-110 (2019)

\bibitem{IS1998} Iwaniec,~T., Sbordone,~C.:
\newblock Riesz transforms and elliptic PDEs with VMO coefficients.
\newblock J. Anal. Math., \textbf{74}, No. 1, 183-212 (1998)

\bibitem{IS2001} Iwaniec,~T., Sbordone,~C.:
\newblock Quasiharmonic fields.
\newblock Ann. I.H.Poincar\'e, Non Linear Analysis, \textbf{18}, No. 5, 519-572 (2001)

\bibitem{JK1995} Jerison,~D., Kenig,~C.~E.:
\newblock The inhomogeneous Dirichlet problem in Lipschitz domains.
\newblock J. Funct. Anal., \textbf{130}, No. 1, 161-219 (1995)

\bibitem{LN} Leonetti,~F., Nesi,~V.:
\newblock Quasiconformal solutions to certain first order systems and the proof of a conjecture of G. W. Milton.
\newblock J. Math. Pures Appl., \textbf{76}, No. 2, 109-124  (1997)

\bibitem{Leoni} Leoni,~G.:
\newblock A First Course in Sobolev Spaces.
\newblock Grad. Stud. Math., vol. 105, Amer. Math. Soc., Providence, RI, (2009)

\bibitem{LZ2022} Li,~C., Zhang,~K.:
\newblock A note on the Gagliardo-Nirenberg inequality in a bounded domain.
\newblock Commun. Pure Appl. Anal. \textbf{21}, No. 12, 4013-4017 (2022)

\bibitem{Maz1961} Maz'ya,~V.~G.:
\newblock Some estimates of solutions of second-order elliptic equations.
\newblock Dokl. Akad. Nauk SSSR, \textbf{137}, No. 5, 1057-1059 (1961)

\bibitem{M} Meyers,~N.~G.:
\newblock An {$L^{p}$}-estimate for the gradient of solutions of second order elliptic divergence equations.
\newblock Ann. Scuola Norm. Sup. Pisa, \textbf{17}, No. 3, 189-206 (1963).

\bibitem{Miranda} Miranda,~C.:
\newblock Sulle equazioni ellittiche del secondo ordine di tipo non variazionale, a coefficienti discontinui.
\newblock Ann. Mat. Pura et Appl., \textbf{63}, No. 1, 353-386 (1963)

\bibitem{PerSem1985} Perelmuter,~M.~A., Semenov,~Yu.~A.:
\newblock Essential self-adjointness of second-order elliptic operator with measurable coefficients.
\newblock Ukr. Math. J., \textbf{37}, No. 2, 163-168 (1985)

\bibitem{Simon1} Simon,~B.:
\newblock A canonical decomposition for quadratic forms with applications to monotone convergence theorems.
\newblock J. Func. Anal., \textbf{28}, No. 3, 377-385 (1978)

\bibitem{Simon2} Simon,~B.:
\newblock Lower semicontinuhy of positive quadratic forms.
\newblock Proc.Royal Soc. Edinburgh: Section A Mathematics. \textbf{79}, No. 3-4, 267-273 (1978).

\bibitem{Stamp} Stampacchia,~G.:
\newblock Le probl\`{e}me de Dirichlet pour les \'equations elliptiques du second ordre \`{a} coefficients discontinus.
\newblock  Ann. Inst. Fourier (Grenoble), \textbf{15}, No. 1,  189-258 (1965)

\bibitem{Stein} Stein,~E.~M.:
\newblock Singular Integrals and Differentiability Properties of Functions.
\newblock Princeton Univ. Press, Princeton, N.J., (1970)

\bibitem{Trudin} Trudinger,~N.~S.:
\newblock Linear elliptic operators with measurable coefficients.
\newblock Ann. Scuola Norm. Sup. Pisa \textbf{27}, No. 2, 265-308 (1973)

\bibitem{Yosida} Yosida, K.:
\newblock Functional Analysis.
\newblock Classics in Mathematics. Springer, Berlin, (1995)

\end{thebibliography}
\end{document}